\documentclass[a4paper,12pt,twoside,leqno,final]{amsart}
\usepackage{amsmath}
\usepackage{amssymb}

\newtheorem{thm}{Theorem}[section]
\newtheorem{lem}[thm]{Lemma}

\newtheorem{prop}[thm]{Proposition}

\newcommand{\C}{{\mathbb C}}
\newcommand{\D}{{\mathbb D}}

\newcommand{\T}{{\mathbb T}}

\newcommand{\La}{\Lambda}

\newcommand{\f}{\frac}
\newcommand{\ov}{\overline}
\newcommand{\al}{\alpha}
\newcommand{\be}{\beta}

\newcommand{\de}{\delta}

\newcommand{\ze}{\zeta}

\newcommand{\ph}{\varphi}
\newcommand{\om}{\omega}

\newcommand{\const}{\text{\rm const}}

\numberwithin{equation}{section}

\title[Functions in Bloch-type spaces and their moduli]
{Functions in Bloch-type spaces\\ 
and their moduli}
\dedicatory{In memory of Victor Petrovich Havin}
\author{Konstantin M. Dyakonov}
\address{ICREA, BGSMath and Universitat de Barcelona, Departament de 
Matem\`atiques i Inform\`atica, Gran Via 585, E-08007 Barcelona, Spain}
\email{konstantin.dyakonov@icrea.cat}
\thanks{Supported in part by grants MTM2011-27932-C02-01, MTM2014-51834-P from El Ministerio de 
Econom\'ia y Competitividad (Spain) and grant 2014-SGR-289 from AGAUR (Generalitat de Catalunya).}

\begin{document}
\begin{abstract} Given a suitably regular nonnegative function $\om$ on $(0,1]$, let $\mathcal B_\om$ denote 
the space of all holomorphic functions $f$ on the unit ball $\mathbb B_n$ of $\C^n$ that satisfy 
$$|\nabla f(z)|\le C\f{\om(1-|z|)}{1-|z|},\qquad z\in\mathbb B_n,$$
with some fixed $C=C_f>0$. We obtain a new characterization of $\mathcal B_\om$ functions in terms of their 
moduli. 
\end{abstract}

\maketitle

\section{Introduction and results}

Let $\mathcal H(\mathbb B_n)$ denote the space of holomorphic functions on the ball 
$$\mathbb B_n:=\{z\in\C^n:\,|z|<1\}$$ 
(we write $|\cdot|$ for the usual Euclidean norm on $\C^n$). Talking about subclasses of $\mathcal H(\mathbb B_n)$, 
or just about function classes in general, we may single out two large families of spaces. First, there 
are {\it growth spaces} defined by imposing an explicit size condition, either integral or uniform, 
on the function's modulus. A growth space $X\subset\mathcal H(\mathbb B_n)$ will typically have 
the \lq\lq lattice property": whenever $f\in X$ and $g\in\mathcal H(\mathbb B_n)$ satisfy 
$|f|\ge|g|$ on $\mathbb B_n$, it follows that $g\in X$. This family contains the classical Hardy and 
Bergman spaces, various weighted $H^\infty$ spaces involving specific majorants on the modulus, etc. 
Secondly, there are {\it smoothness spaces} defined in terms of derivatives and/or differences that are built from 
the function itself (rather than from its modulus). Among the representatives of the latter family are the Lipschitz, 
Besov and Sobolev spaces, to mention a few. 

\par Rather surprisingly, it turns out that a number of (holomorphic) smoothness spaces nevertheless admit a 
fairly explicit description in terms of the moduli of their members. The conditions that arise should, of course, 
govern the oscillations of the function's modulus, not just its growth rate. For Lipschitz spaces, several such 
characterizations were obtained by the author in \cite{DActa} for the case of the disk $\D:=\mathbb B_1$. They 
were subsequently extended in \cite{DAdv} to $\mathbb B_n$, and in fact to more general domains in $\C^n$. We 
also cite \cite{DMich} in connection with holomorphic Besov spaces on $\D$. 

\par The purpose of this note is to provide a similar characterization for certain \lq\lq Bloch-type" spaces that 
result from a growth restriction on the gradient 
$$\nabla f=(\partial_1f,\dots,\partial_nf)$$ 
of a function $f\in\mathcal H(\mathbb B_n)$; here $\partial_j$ stands for the partial differentiation 
operator $\f{\partial}{\partial z_j}$. More precisely, given a (reasonably nice) positive function $\om$ on 
the interval $(0,1]$, the associated Bloch-type space $\mathcal B_\om=\mathcal B_\om(\mathbb B_n)$ consists, 
by definition, of the functions $f\in\mathcal H(\mathbb B_n)$ that obey the condition 
\begin{equation}\label{eqn:defblochtype}
|\nabla f(z)|\le C_f\f{\om(1-|z|)}{1-|z|},\qquad z\in\mathbb B_n,
\end{equation}
with some fixed constant $C_f>0$ on the right. 

\par To be more specific about the class of $\om$'s we have in mind, we now introduce the appropriate terminology. 
We say that a function $\om:(0,1]\to(0,\infty)$ is {\it moderate} if there is a constant $C>0$ with the 
following property: whenever $a\in(0,1]$ and $b\in(0,1]$ satisfy 
$$\f12\le\f ab\le2,$$
we have 
$$\f1C\le\f{\om(a)}{\om(b)}\le C.$$ 
In particular, if $\om$ is a nondecreasing (resp., nonincreasing) positive function such that the ratio 
$\om(2t)/\om(t)$ is bounded above (resp., below) for $0<t\le\f12$, then $\om$ is moderate. 

\par If $\om(t)$ tends to $0$ fast enough as $t\to0^+$, so that $\om(t)=o(t)$, then no nonconstant function 
is in $\mathcal B_\om$. (Obviously, this is not the case we are interested in.) By contrast, $\mathcal B_\om$ 
is sure to be nontrivial once we assume that the function $t\mapsto\om(t)/t$ is nonincreasing. Now, if $\om$ is 
a nondecreasing function with the latter property, and if $\om$ is \lq\lq not too slow" near $0$ in the sense that 
$$\int_0^\de\f{\om(t)}t\,dt\le\const\cdot\om(\de),\qquad0<\de<1$$
(such $\om$'s are called {\it fast majorants} in \cite{DAdv}), then $\mathcal B_\om$ becomes the holomorphic 
{\it Lipschitz space} $\La_\om$ associated with $\om$; its members are precisely the functions whose modulus 
of continuity is dominated by $\om$. The special case $\om(t)=t^\al$ ($0<\al\le1$) corresponds to the classical 
Lipschitz space of order $\al$. 

\par When $\om(t)\equiv1$, the space $\mathcal B_\om$ reduces to the usual {\it Bloch space} $\mathcal B$. The 
{\it little Bloch space} $\mathcal B_0$, formed by the functions $f\in\mathcal H(\mathbb B_n)$ with 
\begin{equation}\label{eqn:deflitbloch}
|\nabla f(z)|\cdot(1-|z|)\to0\quad\text{\rm as}\quad|z|\to1^-,
\end{equation}
can be written as $\bigcup_\om\mathcal B_\om$, where $\om$ runs through the collection of all (moderate) functions 
with $\lim_{t\to0^+}\om(t)=0$. On the other hand, if $\om$ satisfies $\lim_{t\to0^+}\om(t)=\infty$, then the 
associated $\mathcal B_\om$ space is clearly larger than $\mathcal B$ and possesses a kind of \lq\lq negative order" 
smoothness. Furthermore, if $\om(t)$ happens to grow rapidly enough as $t\to0^+$, then $\mathcal B_\om$ becomes a 
growth space, meaning that it can be described by a size condition on $|f|$. For instance, letting $\om(t)=t^{-\be}$ 
with some $\be>0$, one may rewrite \eqref{eqn:defblochtype} in the simpler form 
$$|f(z)|\le\f{\const}{(1-|z|)^\be},\qquad z\in\mathbb B_n.$$ 
However, the case of a milder (say, logarithmic) growth rate of $\om$ near $0$ is more delicate: the resulting 
$\mathcal B_\om$ space is then closer in nature to $\mathcal B$, and it is no longer describable in terms of 
a growth estimate on $|f(z)|$ as $|z|\to1^-$. 

\par Finally, a bit of notation will be needed. For a point $z\in\mathbb B_n$, we put 
$$d_z:=1-|z|$$ 
and let $B_z$ denote the (Euclidean) open ball with center $z$ and radius $d_z/2$. Also, with a function 
$f\in\mathcal H(\mathbb B_n)$ and a point $z\in\mathbb B_n$ we associate the quantity 
$$M_f(z):=\sup\{|f(w)|:\,w\in B_z\}.$$
Next, we introduce the zero set 
$$\mathcal Z_f:=\{\ze\in\mathbb B_n:f(\ze)=0\}$$ 
and define 
$$E_f:=\{z\in\mathbb B_n:\,B_z\cap\mathcal Z_f\ne\emptyset\}.$$ 
Thus $E_f$ can be viewed as a neighborhood of $\mathcal Z_f$; and if $f$ happens to be zero-free (i.e., 
$\mathcal Z_f=\emptyset$), then we also have $E_f=\emptyset$. 

\par Our main result is as follows. When stating it, and later on, we write $E:=E_f$ and 
$E^c:=\mathbb B_n\setminus E_f$. 

\begin{thm}\label{thm:mare} Given $f\in\mathcal H(\mathbb B_n)$ and a moderate function $\om:(0,1]\to(0,\infty)$, 
the following are equivalent. 

\smallskip{\rm(i)} $f\in\mathcal B_\om$. 

\smallskip{\rm(ii)} There is a constant $C_1>0$ such that 
$$\sup\left\{|f(z_1)|-|f(z_2)|:\,z_1,z_2\in B_z\right\}\le C_1\cdot\om(d_z),\qquad z\in\mathbb B_n.$$

\smallskip{\rm(iii)} There is a constant $C_2>0$ such that 
$$\chi_E(z)\cdot M_f(z)+|f(z)|\log\f{M_f(z)}{|f(z)|}\le C_2\cdot\om(d_z),
\qquad z\in\mathbb B_n$$
(the second term on the left being understood as $0$ if $z\in\mathcal Z_f$).

\smallskip{\rm(iv)} There is a constant $C_3>0$ such that 
$$\chi_E(z)\cdot M_f(z)+\chi_{E^c}(z)\cdot|f(z)|\log\f{M_f(z)}{|f(z)|}\le C_3\cdot\om(d_z),
\qquad z\in\mathbb B_n.$$
\end{thm}

\par We emphasize that conditions (ii), (iii) and (iv) above are indeed expressed in terms of $|f|$ alone. 
Conditions (iii) and (iv) are new even in the Lipschitz case, while (ii) is perhaps not too far from what was 
known previously. In fact, the implications (i)$\implies$(ii)$\implies$(iii)$\implies$(iv) are either elementary 
or trivial (or both), so the main effort consists in verifying that the last -- and formally weakest -- condition 
on $|f|$ is actually sufficient to ensure that $f\in\mathcal B_\om$. In a sense, (iv) seems to be the weakest 
possible condition on the modulus that does the job. We also remark that conditions (iii) and (iv), which 
invlove \lq\lq logarithmic oscillations", are usually easier to check than (ii) or similar Lipschitz-type 
conditions on $|f|$, such as those that arose in \cite{DActa,DAdv} in the $\La_\om$ setting. 

\par Given a nonvanishing function $f\in\mathcal H(\mathbb B_n)$, we have $E=\emptyset$ and $E^c=\mathbb B_n$, 
in which case conditions (iii) and (iv) become the same. Each of these reduces to saying that 
\begin{equation}\label{eqn:logosc}
|f(z)|\log\f{M_f(z)}{|f(z)|}\le\const\cdot\om(d_z),\qquad z\in\mathbb B_n.
\end{equation}

\par To see a consequence of this criterion, let us now consider a nonvanishing holomorphic function $F$ on 
the unit disk $\D:=\mathbb B_1$ in $\C$. (In what follows, we also use the notation $\T:=\partial\D$ for the 
unit circle, and $m$ for the normalized arclength measure on $\T$.) Assume, in addition, that $F$ lies in some 
Hardy space $H^p=H^p(\D)$ with $0<p\le\infty$. It is well known (see, e.g., \cite[Chapter II]{G}) that, except 
for a constant factor of modulus $1$, such a function is necessarily of the form 
\begin{equation}\label{eqn:genform} 
F(z)=\exp\left(\int_\T\f{\ze+z}{\ze-z}\,d\nu(\ze)\right),\qquad z\in\D, 
\end{equation}
$\nu$ being a signed measure on $\T$. This measure can further be written as 
\begin{equation}\label{eqn:measurenu}
d\nu=\log\psi\,dm-d\mu_s,
\end{equation}
where $\psi$ is a nonnegative function satisfying 
\begin{equation}\label{eqn:condpsi}
\psi\in L^p(\T,m)\quad\text{\rm and}\quad\log\psi\in L^1(\T,m)
\end{equation}
(in fact, $\psi(\ze)=\lim_{r\to1^-}|F(r\ze)|$ for $m$-almost every $\ze\in\T$), while $\mu_s$ is a finite positive 
measure on $\T$ singular with respect to $m$. When $\mu_s=0$, $F$ becomes the {\it outer function} with modulus $\psi$, 
whereas taking $\psi\equiv1$ one gets the {\it singular inner function} associated with $\mu_s$; again, we refer 
to \cite[Chapter II]{G} for these matters. 

\par It is straightforward to verify that 
$$\log|F(z)|=\int_\T\f{1-|z|^2}{|\ze-z|^2}\,d\nu(\ze)=:\mathcal P\nu(z)$$ 
for $z\in\D$; here $\mathcal P$ stands for the Poisson integral operator. Therefore, the next result comes out 
readily upon applying \eqref{eqn:logosc} to $F$ in place of $f$. The space $\mathcal B_\om$ in the statement below 
should be understood as $\mathcal B_\om(\D)$. Likewise, the symbols $d_z$ and $B_z$ will have the same meaning 
as before, but restricted to dimension $n=1$. 

\begin{thm}\label{thm:secre} Suppose $F$ is defined by \eqref{eqn:genform}, with $\nu$ as in \eqref{eqn:measurenu}. 
Given $p\in(0,\infty]$ and a moderate function $\om:(0,1]\to(0,\infty)$, the following are equivalent. 
\par{\rm(a)} $F\in\mathcal B_\om\cap H^p$. 
\par{\rm(b)} $\psi$ satisfies \eqref{eqn:condpsi}, and there is a constant $C>0$ such that 
$$\exp\left(\mathcal P\nu(z)\right)\cdot\left[\mathcal P\nu(w)-\mathcal P\nu(z)\right]\le C\cdot\om(d_z)$$
whenever $z\in\D$ and $w\in B_z$. 
\end{thm}

\par In particular, letting $\mu_s=0$, one arrives at a criterion for an outer function to be in 
$\mathcal B_\om\cap H^p$. In the Lipschitz case, when $\mathcal B_\om(=\mathcal B_\om\cap H^p)=\La_\om$, the 
result is also new and supplements previous characterizations of the outer functions in $\La_\om=\La_\om(\D)$ 
that were found by Shirokov \cite{ShiTr, Shi} and by the author \cite{DActa, DJAM}. On the other hand, 
letting $\psi\equiv1$ (and $p=\infty$, say), one obtains a description of the singular inner functions in 
$\mathcal B_\om$; the class of such singular inner functions is nontrivial when $\om(t)$ tends to $0$ slowly 
enough as $t\to0^+$. 

\par One might also consider the \lq\lq little oh" analogues of the $\mathcal B_\om$ spaces and come up with 
the corresponding \lq\lq little oh" version of the theorem above. We restrict ourselves to mentioning the case 
of $\mathcal B_0\cap H^\infty$, where $\mathcal B_0=\mathcal B_0(\D)$ is the little Bloch space defined by 
condition \eqref{eqn:deflitbloch}, this time with $|f'(z)|$ in place of $|\nabla f(z)|$. 

\begin{prop}\label{prop:littlebloch} Suppose $F$ is a function of the form \eqref{eqn:genform}, with $\nu$ as 
in \eqref{eqn:measurenu}. In order that $F\in\mathcal B_0\cap H^\infty$, it is necessary and sufficient 
that $\psi\in L^\infty(\T,m)$ and 
$$\sup\left\{\exp\left(\mathcal P\nu(z)\right)\cdot\left[\mathcal P\nu(w)-\mathcal P\nu(z)\right]:\,
w\in B_z\right\}\to0$$ 
as $|z|\to1^-$. 
\end{prop} 

\par It might be interesting to compare this with Bishop's characterization of $\mathcal B_0\cap H^\infty$, 
as given in \cite{B}. 

\par Postponing the proof of Theorem \ref{thm:mare} to the last section, we shall begin by establishing 
a preliminary result (see Section 2 below), namely a certain Schwarz--Pick type lemma, to lean upon.  
The idea of using this kind of technique for similar purposes goes back to Pavlovi\'c's paper \cite{P}, 
where the classical Schwarz(--Pick) lemma was employed to give a simple proof of the author's earlier 
result from \cite{DActa} on the moduli of holomorphic Lipschitz functions. 
Here, we use a refined version of the Schwarz--Pick inequality that is valid for nonvanishing 
functions only (even though the function $f$ of Theorem \ref{thm:mare} may have zeros). 
This allows us to arrive at the required estimate on $|\nabla f(z)|$ for $z\notin E_f$, while the case 
of $z\in E_f$ is treated separately; see the proof of the (iv)$\implies$(i) part in Section 3.

\section{A Schwarz--Pick type lemma for nonvanishing functions} 

The familiar Schwarz--Pick lemma (see, e.g., \cite[Chapter I]{G}) tells us that if $g$ is a holomorphic 
self-map of the unit disk $\D$ (in $\C$), then 
$$|g'(z)|\le\f{1-|g(z)|^2}{1-|z|^2}$$ 
for all $z\in\D$. See also \cite[Chapter 8]{R} for extensions of this to $\mathbb B_n$ with $n>1$. 
Now, it turns out that if $g$ happens to be zero-free, then a better estimate is possible; 
the refinement is given (in the $\mathbb B_n$ setting) by Lemma \ref{lem:spnonvan} below. In the case 
of $\D$, the result is essentially known. For instance, it can be deduced from the generalized Schwarz--Pick 
lemma due to Ahlfors; see Theorem 1-7 in \cite[Chapter 1]{A}. However, since the required version -- which 
should also work for $\mathbb B_n$ -- does not seem to be readily available in the literature, we have 
chosen to provide a complete self-contained proof thereof. 

\begin{lem}\label{lem:spnonvan} Suppose $g\in\mathcal H(\mathbb B_n)$ is a function satisfying $0<|g(z)|\le1$ 
for all $z\in\mathbb B_n$. Then 
\begin{equation}\label{eqn:schpi}
|\nabla g(z)|\le\f2{1-|z|^2}\,|g(z)|\log\f{1}{|g(z)|},\qquad z\in\mathbb B_n.
\end{equation}
In particular, 
\begin{equation}\label{eqn:schpizero}
|\nabla g(0)|\le2|g(0)|\log\f1{|g(0)|}.
\end{equation}
\end{lem}

\begin{proof} First let us consider the case $n=1$. Thus, $g$ is currently supposed to be a holomorphic 
function on the disk $\D:=\mathbb B_1$ satisfying $0<|g|\le1$ there. We may furthermore assume that $g$ is an 
outer function. (Otherwise, replace $g$ by $g_r$ with $0<r<1$, where $g_r(z):=g(rz)$, and then let $r\to1^-$.) 
This last assumption means that the (nonpositive) harmonic function $h:=\log|g|$ is the Poisson integral 
of its boundary values: 
\begin{equation}\label{eqn:poisson}
h(z)=\int_\T\f{1-|z|^2}{|\ze-z|^2}h(\ze)dm(\ze)=:\mathcal P\left(h|_\T\right)(z),\qquad z\in\D,
\end{equation}
while $g$ itself is of the form 
\begin{equation}\label{eqn:outer}
g(z)=\exp\left(\int_\T\f{\ze+z}{\ze-z}h(\ze)dm(\ze)\right),\qquad z\in\D.
\end{equation}
Differentiating \eqref{eqn:outer} gives 
\begin{equation}\label{eqn:diffouter}
g'(z)=g(z)\cdot U(z),
\end{equation}
where 
$$U(z):=\int_\T\f{2\ze}{(\ze-z)^2}h(\ze)dm(\ze).$$ 
Now, since $0\le -h=|h|$ almost everywhere on $\T$, we have 
\begin{equation*}
\begin{aligned}
|U(z)|&\le -\int_\T\f2{|\ze-z|^2}h(\ze)dm(\ze)\\
&=-\f2{1-|z|^2}\mathcal P\left(h|_\T\right)(z)\\
&=-\f2{1-|z|^2}h(z),
\end{aligned}
\end{equation*}
where the last step relies on \eqref{eqn:poisson}. In conjunction with \eqref{eqn:diffouter}, this yields 
\begin{equation*}
\begin{aligned}
|g'(z)|&=|g(z)|\cdot|U(z)|\\
&\le-\f{2|g(z)|}{1-|z|^2}h(z)\\
&=\f{2|g(z)|}{1-|z|^2}\log\f1{|g(z)|}.
\end{aligned}
\end{equation*}
We have thereby established \eqref{eqn:schpi}, and in particular \eqref{eqn:schpizero}, in dimension $n=1$. 
\par Our next step is to prove \eqref{eqn:schpizero} in the case $n>1$. Assuming that $\nabla g(0)\ne0$ (otherwise 
the inequality is trivial), we consider the unit vector 
$$\ze=\nabla g(0)/|\nabla g(0)|=\f1{|\nabla g(0)|}\left(\partial_1g(0),\dots,\partial_n g(0)\right)$$ 
and put 
$$G(w):=g(w\ov\ze),\qquad w\in\D.$$ 
Because $G$ is a holomorphic function on $\D$ with $0<|G|\le1$ and 
$$G'(0)=\langle\nabla g(0),\ze\rangle=|\nabla g(0)|,$$ 
where $\langle\cdot,\cdot\rangle$ denotes the usual inner product in $\C^n$, the (already known) inequality 
$$|G'(0)|\le2|G(0)|\log\f1{|G(0)|}$$ 
reduces to \eqref{eqn:schpizero}; the latter is thus established for every $n$. 
\par Finally, to prove \eqref{eqn:schpi} in full generality, we fix a nonzero point $a\in\mathbb B_n$ and 
consider the automorphism $\ph_a$ of $\mathbb B_n$ that interchanges $a$ and $0$. This is given by 
$$\ph_a(z)=\f{a-P_az-(1-|a|^2)^{1/2}Q_az}{1-\langle z,a\rangle},\qquad z\in\mathbb B_n,$$ 
where $P_a$ is the orthogonal projection of $\C^n$ onto the one-dimensional subspace spanned by $a$, and 
$Q_a=I-P_a$ is the complementary projection. Then we define 
\begin{equation}\label{eqn:pozorpes}
F(z):=(g\circ\ph_a)(z),\qquad z\in\mathbb B_n,
\end{equation}
so that $F$ is a holomorphic function on $\mathbb B_n$ satisfying $0<|F|\le1$ there. An application of 
\eqref{eqn:schpizero}, with $F$ in place of $g$, yields 
\begin{equation}\label{eqn:schpizerofff}
|\nabla F(0)|\le2|F(0)|\log\f1{|F(0)|}.
\end{equation} 
Differentiating \eqref{eqn:pozorpes} and bearing in mind that $\ph_a\circ\ph_a$ is the identity map, we find that 
\begin{equation}\label{eqn:golopuzoe}
\nabla g(a)=\nabla F(0)\cdot\left(\ph'_a(0)\right)^{-1}=\nabla F(0)\cdot\ph'_a(a), 
\end{equation}
where the gradients are interpreted as row vectors, while $\ph'_a$ stands for the appropriate Jacobian 
matrix. The formula 
$$\ph'_a(a)=-(1-|a|^2)^{-1}P_a-(1-|a|^2)^{-1/2}Q_a$$ 
(see \cite[Section 2.2]{R}) implies readily that $\|\ph'_a(a)\|$, the operator norm of the matrix $\ph'_a(a)$, 
is bounded by $(1-|a|^2)^{-1}$. It now follows from \eqref{eqn:golopuzoe} that 
\begin{equation}\label{eqn:golozadoe}
|\nabla g(a)|\le\f1{1-|a|^2}\cdot|\nabla F(0)|.
\end{equation}
Finally, we notice that $F(0)=g(a)$ and combine \eqref{eqn:schpizerofff} with \eqref{eqn:golozadoe} to obtain 
$$|\nabla g(a)|\le\f2{1-|a|^2}|g(a)|\log\f1{|g(a)|}.$$ 
This is precisely \eqref{eqn:schpi}, with $a$ in place of $z$, and we are done. 
\end{proof}

\section{Proof of Theorem \ref{thm:mare}}

(i)$\implies$(ii). Fix $z\in\mathbb B_n$ and let $z_1,z_2\in B_z$. Clearly, 
\begin{equation}\label{eqn:gusi}
|f(z_1)|-|f(z_2)|\le|f(z_1)-f(z_2)|\le\int_{[z_1,z_2]}|\nabla f(\ze)|\,|d\ze|,
\end{equation}
where $[z_1,z_2]$ denotes the segment with endpoints $z_1$ and $z_2$. Since $f\in\mathcal B_\om$, we have 
\begin{equation}\label{eqn:twoineq}
|\nabla f(\ze)|\le C\f{\om(d_\ze)}{d_\ze}\le C_1\f{\om(d_z)}{d_z},\qquad\ze\in[z_1,z_2],
\end{equation}
where $C$ and $C_1$ are suitable constants. (The last inequality in \eqref{eqn:twoineq} is due to 
the fact that $\f12d_z\le d_\ze\le\f32d_z$ for $\ze\in B_z$, combined with the hypothesis on $\om$.) Using 
\eqref{eqn:twoineq} to estimate the integral in \eqref{eqn:gusi}, while noting that the length of $[z_1,z_2]$ 
is at most $d_z$, we obtain 
$$|f(z_1)|-|f(z_2)|\le C_1\om(d_z),$$ 
which proves (ii). 

\par (ii)$\implies$(iii). Let $z\in E(=E_f)$, so that $B_z$ contains a point $z_0$ with $f(z_0)=0$. For $w\in B_z$, 
(ii) yields 
$$|f(w)|=|f(w)|-|f(z_0)|\le C_1\om(d_z),$$
whence 
\begin{equation}\label{eqn:lobok}
M_f(z)\le C_1\om(d_z),\qquad z\in E.
\end{equation}
On the other hand, (ii) tells us that 
$$M_f(z)-|f(z)|\le C_1\om(d_z),\qquad z\in\mathbb B_n,$$ 
and we combine this with the elementary inequality 
$$a\log\f{b}{a}\le b-a\qquad(0<a<b)$$ 
to deduce that 
\begin{equation}\label{eqn:kolobok}
|f(z)|\log\f{M_f(z)}{|f(z)|}\le C_1\om(d_z),\qquad z\in\mathbb B_n.
\end{equation}
Finally, \eqref{eqn:lobok} and \eqref{eqn:kolobok} together imply (iii). 

\par (iii)$\implies$(iv). This is obvious. 

\par (iv)$\implies$(i). Let $z\in E$. From (iv) we know that 
$$M_f(z)\le C\cdot\om(d_z),$$ 
with some fixed $C>0$, and we deduce (e.g., by applying the Cauchy formula to a suitable polydisk centered 
at $z$ and contained in $B_z$) that 
$$|\partial_jf(z)|\le\const\cdot\f{\om(d_z)}{d_z}\qquad(j=1,\dots,n).$$
Hence 
\begin{equation}\label{eqn:pupochek}
|\nabla f(z)|\le\const\cdot\f{\om(d_z)}{d_z},
\end{equation}
possibly with another constant on the right. 
\par Now assume that $z\in E^c(=\mathbb B_n\setminus E)$, so that $f$ has no zeros in $B_z$. The function 
$$g_z(w):=f\left(z+\f{d_z}2w\right)/M_f(z),\qquad w\in\mathbb B_n,$$
is then zero-free and bounded in modulus by $1$ on $\mathbb B_n$. An application of Lemma \ref{lem:spnonvan} gives 
$$|\nabla g_z(0)|\le2|g_z(0)|\log\f1{|g_z(0)|},$$
or equivalently, 
\begin{equation}\label{eqn:sosochek}
\f{d_z}2|\nabla f(z)|\le2|f(z)|\log\f{M_f(z)}{|f(z)|}.
\end{equation}
By (iv), there is a $C>0$ such that 
$$|f(z)|\log\f{M_f(z)}{|f(z)|}\le C\cdot\om(d_z),$$ 
and combining this with \eqref{eqn:sosochek} we arrive at \eqref{eqn:pupochek}, this time for $z\in E^c$. 
Thus \eqref{eqn:pupochek} actually holds for all $z\in\mathbb B_n$, and the proof is complete.

\end{document}